\documentclass[
	paper=A4,
	pagesize=auto,
	fontsize=12pt,
	DIV = 14,
	oneside
]{scrartcl}
\usepackage{libertine}

\usepackage{amsmath,amssymb,mathtools,bm}

\usepackage[UKenglish]{babel}
\usepackage[utf8]{inputenc}

\usepackage{hyperref}
\usepackage{csquotes}

\usepackage{amstext}
\usepackage{amsthm}

\usepackage{enumerate}
\usepackage{hyperref}

\usepackage[nameinlink,capitalise]{cleveref}

\usepackage{xcolor}
\usepackage{cancel}
\usepackage[normalem]{ulem}

\DeclareMathOperator{\Erf}{\mathrm{Erf}}
\DeclareMathOperator{\Erfc}{\mathrm{Erfc}}

\renewenvironment{thebibliography}[1]{
	\begin{oldthebibliography}{#1}
		\footnotesize
		\setlength{\itemsep}{0em}
		\setlength{\parskip}{0em}
	}
	{
	\end{oldthebibliography}
}

\numberwithin{equation}{section}
\newtheorem{theorem}{Theorem}[section]
\newtheorem{lemma}[theorem]{Lemma}

\newtheorem{corollary}[theorem]{Corollary}

\theoremstyle{definition}

\theoremstyle{remark}
\newtheorem{remark}[theorem]{Remark}
\newtheorem{example}[theorem]{Example}

\crefname{equation}{}{}
\crefname{enumi}{}{}
\Crefname{enumi}{}{}
\creflabelformat{enumi}{#2(#1)#3}

\begin{document}
\title{Super-Gaussian Decay of Exponentials:\\ A Sufficient Condition}
\author{Benjamin Hinrichs\thanks{Friedrich Schiller University Jena, Department of Mathematics, Ernst-Abbe-Platz 2, 07743 Jena, Germany\\ Present Affiliation: Paderborn University, Institute of Mathematics, Institute for Photonic Quantum Systems, Warburger Str. 100, 33098 Paderborn, Germany}
	\qquad
	Daan W. Janssen\thanks{University of Leipzig, Institute for Theoretical Physics, Brüderstra\ss e 16, 04103 Leipzig, Germany}
	\qquad
	Jobst Ziebell\thanks{Friedrich Schiller University Jena, Institute for Theoretical Physics, Max-Wien-Platz 1, 07743 Jena, Germany}}
\newcommand{\shortauthors}{B. Hinrichs, D. W. Janssen, J. Ziebell}
\newcommand{\titlename}{Super-Gaussian Decay of Exponentials}
\date{}
\maketitle

\begin{abstract}\noindent
In this article, we present a sufficient condition for the exponential $\exp({-f})$ to have a tail decay stronger than any Gaussian, where $f$ is defined on a locally convex space $X$ and grows faster than a squared seminorm on $X$.
In particular, our result proves that $\exp({-p(x)^{2+\varepsilon}+\alpha q(x)^2})$ is integrable for all $\alpha,\varepsilon>0$ w.r.t. any Radon Gaussian measure on a nuclear space $X$, if $p$ and $q$ are continuous seminorms on $X$ with compatible kernels. This can be viewed as an adaptation of Fernique's theorem and, for example, has applications in quantum field theory.
\end{abstract}

\section{Introduction}
A well-known theorem by Fernique \cite{src:Fernique} states that given a Gaussian measure $\mu$ on a locally convex space $X$ along with a measurable seminorm $q$ on $X$ there is some $\alpha > 0$ such that
\begin{equation}
	\int_X \exp \left[ \alpha q \left( x \right)^2 \right] \mathrm{d} \mu \left( x \right) < \infty \, .
\end{equation}
This is an extremely useful property that has been generalised in many directions, see for example \cite{src:AidaMasudaShigekawa,src:Basse,src:FrizOberhauser}.
In this article, we give a generalisation in a somewhat different direction that was motivated by a problem in quantum field theory:
Suppose one has a regularised free scalar quantum field theory modelled by a Gaussian measure $\mu$ on the space $\mathcal{S} (\mathbb{R}^d)$ of Schwartz functions and an interaction term given by $\lambda \int_{\mathbb{R}^d} \phi^4$ for some $\lambda > 0$.
Then one typically attempts to create a \enquote{counterterm} proportional to $\int_{\mathbb{R}^d} \phi^2$ in order to absorb divergences that appear upon the undoing of the regularisation.
Mathematically, this corresponds to controlling integrals of the form
\begin{equation}
	\label{eq:IntegralWithCounterterm}
	\int_{\mathcal{S} (\mathbb{R}^d)} \exp \left[ -\lambda \int_{\mathbb{R}^d} \phi^4 + \alpha \int_{\mathbb{R}^d} \phi^2 \right] \mathrm{d} \mu \left( \phi \right)
\end{equation}
for arbitrarily large $\alpha > 0$.
Note that the above integrand is unbounded, by the inequivalence of the $L^4$- and $L^2$-norms on $\mathcal{S} (\mathbb{R}^d)$.
By Fernique's theorem, \eqref{eq:IntegralWithCounterterm} is certainly finite whenever $\alpha$ is small enough, but for large $\alpha$ the integrability is unclear.
This motivates the following simple question:

\bigskip\noindent
{\em Given a Gaussian measure $\mu$ on a locally convex space $X$ along with two measurable seminorms $p$ and $q$ as well as some $\varepsilon > 0$, under what conditions is
	\begin{equation}
		\label{eq:theIntegral}
		\int_X \exp \left[ -p \left( x \right)^{2 + \varepsilon} + \alpha q \left( x \right)^{2} \right] \mathrm{d} \mu \left( x \right)
	\end{equation}
	finite for all $\alpha > 0$?}

\bigskip\noindent
In the finite-dimensional setting the natural integral measure would be the Lebesgue measure for which it is clear that the corresponding integral is finite if and only if $p$ is equivalent to the Euclidean norm.
In that case one can majorise $q$ by a multiple of the Euclidean norm and retain the finiteness of the integral.
However, since we ask for arbitrary $\alpha > 0$ there is no loss of generality by absorbing a Gaussian density $\exp \left[ - \left \Vert x \right \Vert^2 / 2 \right]$ into the measure, thereby making it Gaussian.
Hence, our question provides a natural generalisation from finite dimensional vector spaces to arbitrary locally convex spaces.

Using basic properties in the theory of absolutely summing operators \cite{src:Pietsch:AbsSummingHilbertSchmidtAndPrenuclear,src:Pietsch:pSummingMeasureCharacterisation}, we find simple sufficient conditions that turn out to be especially applicable to Gaussian measures on nuclear spaces.
The most fundamental requirement will be that $p = 0$ has to imply $q = 0$ on a large enough completion of the Cameron--Martin space of $\mu$.
Assuming this, we may then bound \eqref{eq:theIntegral} by integrals against finite-dimensional projections of $\mu$.
These projected integrals are uniformly bounded provided that the inclusion map from the Cameron--Martin space to the aforementioned completion satisfies a summability property.
We will also demonstrate that a yet slightly stronger requirement on the inclusion map guarantees the existence of an orthonormal basis of the Cameron--Martin space for which the finite-dimensional approximations in fact converge to \eqref{eq:theIntegral}.
Again, this turns out to be especially applicable to nuclear spaces.

This paper is organized as follows. In \cref{sec:Gauss},
we introduce well-known notions and statements from the theory of Gaussian measures.
In the end of this \lcnamecref{sec:Gauss}, we also discuss the question posed above in view of two simple examples.
In \cref{sec:psum}, we then introduce absolutely $p$-summing operators
and prove some simple extensions of results from the literature.
In \cref{sec:integrability}, we 
prove the integrability result for continuous seminorms on locally convex spaces.
Then we formulate our result on finite-dimensional approximations in a metrizable setting, in \cref{sec:UniformIntegrability}. 
Finally, \Cref{sec:MeasurableSeminorms} generalizes the results to measurable seminorms on arbitrary locally convex spaces and provides another application thereof.
\section{Gaussian Measures}\label{sec:Gauss}
The conventions and definitions given here follow those used in \cite{src:Bogachev:GaussianMeasures}.

All \textbf{vector spaces} are taken to be real and locally convex spaces are assumed to be Hausdorff.
The (topological) dual space of a topological vector space $X$ will be denoted by $X^*$ and if $X$ is a normed space $X^*$ will be assumed to be equipped with its induced norm, turning it into a Banach space.
Given a seminorm $q$ on a vector space $V$, we also denote by $q$ the induced norm on the quotient space $V / \ker q$.
Furthermore, we define $V_q$ to be the Banach completion of $V / \ker q$ with respect to $q$ such that $q$ also denotes the norm on $V_q$.
The linear map $V \to V_q$ given by the composition of the canonical map $V / \ker q \to V_q$ and the quotient map $V \to V / \ker q$ will (in lack of a better name) be referred to as the \textbf{natural map} from $V$ to $V_q$.
Whenever there is another seminorm $p \ge q$ on $V$, $q$ induces a continuous seminorm on $V_p$ which we again will refer to as $q$ such that the natural map factorises like $V \to V_p \to V_q$.
Consequently, we shall also speak of the natural map $V_p \to V_q$.
All \textbf{measures} are taken to be countably additive and non-negative.
For a measure space $(X,\mathcal A,\mu)$, we denote the Lebesgue completion of $\mathcal A$ w.r.t. $\mu$ as $\mathcal A_\mu$.
Given a topological space $X$, a finite measure $\mu$ on its Borel $\sigma$-algebra $\mathcal B(X)$ is called \textbf{Radon} if, for every Borel set $B \subseteq X$ and every $\varepsilon > 0$, there exists a compact set $K \subseteq B$ such that $\mu \left( B \setminus K \right) < \varepsilon$.
A Borel probability measure $\mu$ is called \textbf{Gaussian} if for all $f \in X^*$ the induced measure $\mu \circ f^{-1}$ on $\mathbb{R}$ is Gaussian or a Dirac measure.
It is called \textbf{centred} if all corresponding measures $\mu \circ f^{-1}$ are centred.

We will integrate unbounded continuous functions, which we treat in the next {well-known} \lcnamecref{lem:weakconvunbound}.
We present a short proof for the convencience of the reader. 
\begin{lemma}
	\label{lem:weakconvunbound}
	Let $X$ be a topological space, $(\mu_n)$ a sequence of Radon measures on $\mathcal B(X)$ weakly converging to a Radon measure $\mu$ and $f:X\to[0,\infty)$ a continuous function. Then
	\begin{equation} \int_X f \mathrm d \mu \le \liminf_{n\to\infty} \int_X f \mathrm d\mu_n.  \end{equation}
\end{lemma}
\begin{proof}
	We define the continuous and bounded functions $f_k = \min\{f,k\}$ for $k\in\mathbb N$. By definition, we have
	\begin{equation} \int_X f_k \mathrm d \mu = \lim_{n\to\infty} \int_X f_k \mathrm d \mu_n \le \liminf_{n\to\infty} \int_X f \mathrm d \mu_n \qquad\mbox{for all}\ k\in \mathbb N. \end{equation}
	Furthermore, by the monotone convergence theorem, the left-hand side converges to $\int_X f\mathrm d \mu$ as $k\to \infty$. This proves the statement.
\end{proof}
We investigate \textbf{seminorms} on locally convex spaces $X$ equipped with a Radon Gaussian measure $\mu$.
A function $q$ on $X$ is called a \textbf{$\mu$-measurable seminorm} if there exists a $\mathcal{B}(X)_\mu$-measurable linear subspace $X_0$ of $X$ such that $\mu(X_0) = 1$ and $q \upharpoonright X_0$ is a seminorm on $X_0$.

In the theory of Gaussian measures, it is natural to study the \textbf{Cameron--Martin space}, which we briefly introduce here and refer to \cite{src:Bogachev:GaussianMeasures} for further details.
Given a centred Gaussian measure $\mu$ on a locally convex space $X$, the Cameron--Martin space $H(\mu)$ is defined as the set of all $h \in X$ with finite Cameron--Martin norm
\begin{equation}
	\left \Vert h \right \Vert_{H(\mu)} = \sup \left \{
	\phi \left( h \right) \, \middle| \, \phi \in X^* : \int_X |\phi(x)|^2 \mathrm{d} \mu \left( x \right) \le 1
	\right \} \, .
\end{equation}
We remark that the integral is always finite, since the definition of Gaussian measures implies $X^*\subseteq L^2(\mu)$ and denote the closure of $X^*$ in $L^2(\mu)$ with $X^*_\mu$.
The inclusion map $H(\mu) \to X$ is continuous by \cite[Proposition 2.4.6]{src:Bogachev:GaussianMeasures}.
\begin{lemma}[{\cite[Proposition 2.4.6, Theorem 3.2.7, Theorem 3.6.1]{src:Bogachev:GaussianMeasures}}]\ \label{lem:CMhilbert}
	If $\mu$ is a centred Radon Gaussian measure on a locally convex space $X$, then $H(\mu)$ is a separable Hilbert space and $\mu$ is supported on the closure of $H(\mu)$ in $X$.
\end{lemma}
We define the linear operator $R_\mu$ from $X^*_\mu$ to the algebraic dual space of $X^*$ by
\begin{equation}
	R_\mu \left( \phi \right) \left( \psi \right)
	=
	\int_X \phi \left( x \right) \psi \left( x \right) \mathrm{d} \mu \left( x \right)
\end{equation}
for all $\phi \in X^*_\mu$ and $\psi \in X^*$.
In many cases $R_\mu (\phi)$ is generated by an element $x\in X$ in the sense that $R_\mu(\phi)(\psi) = \psi(x)$ for all $\psi\in X^*$. In this case, we write $x=R_\mu(\phi)$.
In this sense, by \cite[Lemma 2.4.1]{src:Bogachev:GaussianMeasures}, we have the identity
\begin{equation}\label{eq:CM=Rmu}
	H(\mu) = R_\mu(X_\mu^*) \cap X
	\quad\mbox{and}\quad 
	\|R_\mu f\|_{H(\mu)} = \|f\|_{L^2(\mu)}
	\ \mbox{for}\ f\in R_\mu^{-1}(X).
\end{equation}
Furthermore, if $f \in R_\mu^{-1}(X)$ is linear, then
\begin{equation}\label{eq:Xmulinear}
	f(h) = \langle R_\mu f,h \rangle_{H(\mu)} \qquad\mbox{for all}\ h\in H(\mu).
\end{equation}
In particular, for any centred Radon Gaussian measure $\mu$, $R_\mu$ is a Hilbert isomorphism from $X^*_\mu$ to $H(\mu)$, see \cite[Theorem 3.2.3, Theorem 3.2.7]{src:Bogachev:GaussianMeasures}.

The importance of measurable seminorms arises partially because these play well together with the Cameron--Martin space.
\begin{theorem}[{\cite[Theorem 3.2.10(i)]{src:Bogachev:GaussianMeasures}}]
	\label{thm:MeasurableSeminormContinuousOnCameronMartin}
	Let $\mu$ be a Radon Gaussian measure on a locally convex space $X$ and $q$ a $\mu$-measurable seminorm. 
	Then the restriction of $q$ to the Cameron--Martin space $H(\mu)$ is continuous.
\end{theorem}
Let us also recall the Cameron--Martin theorem.
\begin{theorem}[{\cite[Corollary 2.4.3, Remark 3.1.8]{src:Bogachev:GaussianMeasures}}]
	\label{thm:CameronMartin}
	Let $\mu$ be a centred Radon Gaussian measure on a locally convex space $X$, $h \in H(\mu)$ and $\tau_h : X \to X, x \mapsto x - h$.
	Then $\mu$ and $\mu \circ \tau_h^{-1}$ are equivalent with density
	\begin{equation}
		\frac{\mathrm{d} \mu \circ \tau_h^{-1}}{\mathrm{d} \mu} \left( x \right)
		=
		\exp \left[ R_\mu^{-1} \left( h \right) \left( x \right) - \frac{1}{2} \left\Vert h \right\Vert_{H(\mu)}^2 \right] \, .
	\end{equation}
	Furthermore, any $x \in X$ with the property that $\mu$ and $\mu \circ \tau_x^{-1}$ are equivalent lies in $H(\mu)$.
\end{theorem}
Let us finish this \lcnamecref{sec:Gauss}, by some examples motivating our investigation.
The first \lcnamecref{ex1} is a counterexample, illustrating that it is not guaranteed that integrals of the form \cref{eq:theIntegral} are finite.
It is based upon \cite[Examples 2.8.2 \& 2.8.9]{src:Bogachev:GaussianMeasures}.
\begin{example}\label{ex1}
	Let $\mu$ be the product of the standard Gaussian measures on $X=\mathbb{R}^{\mathbb{N}}$
	and let
	\[
	p(x)=\limsup_{n\to\infty}\left(\frac 1n \sum_{i=1}^{n}x_i^2\right)^{1/2}
	\quad\mbox{and}\quad 
	q(x)=|x_1|.
	\]
	Both definitions provide $\mu$-measurable seminorms, since $q$ defines a seminorm on all of $X$
	and $p$ is a seminorm on the linear subspace of $X$ on which it is finite.
	Furthermore, the law of large numbers implies that $\mu(\{p=1\})=1$.
	For $\alpha,\varepsilon>0$, employing the product measure structure of $\mu$, we can now explicitly calculate
	\begin{align*}
		\int_X \exp({-p^{2+\varepsilon}+\alpha q^2}) \mathrm{d} \mu(x)
		& = e^{-1} \int_X \exp( \alpha q^2)\mathrm{d}\mu
		\\&
		= (\sqrt{2\pi}e)^{-1}\int_{\mathbb{R}}\exp((\alpha-\tfrac 12)|x_1|^2) \mathrm{d}x_1.
	\end{align*}
	This is finite if and only if $\alpha<\tfrac 12$.
\end{example}
Obviously, in above \lcnamecref{ex1}, we could replace $q$ by any larger seminorm with the same result.
This especially includes any $\ell^r$-norm with $r\in[1,\infty]$.

Let us now consider the case in which $p$ is the $\ell^2$-norm and $q$ is a different $\ell^r$-norm.
\begin{example}\label{exsimple}
	Let $(\sigma_k)_{k\in\mathbb N}\subset (0,\infty)$ and let $\mu_\sigma$ be the Gaussian measure on $X=\mathbb R^{\mathbb N}$ given as the product measure of centred Gaussian measures $\mu_k$ with variance $\sigma_k$ on $\mathbb R$.
	Now, we define the usual $\ell^r$-norms as $p_r(x) = \limsup_{n\to\infty}\left(\sum_{k=1}^{n}x_k^r\right)^{1/r}$ for $r\in[1,\infty)$ and $p_\infty(x) = \limsup_{n\to\infty} \sup_{k=1,\ldots,n}|x_k|$.
	We recall that $p_r\le p_s$ if $r\ge s$.
	If $\sigma \in \ell^2$, then $\mu_\sigma(\ell^2) = 1$, cf. \cite[Example 2.3.6]{src:Bogachev:GaussianMeasures}, so $p_r$ defines a $\mu_\sigma$-measurable seminorm on $X$ for all $r\ge 2$ in this case. Furthermore, for all $\alpha,\varepsilon>0$, this implies
	\[ \int_X \exp( -p_2^{2+\varepsilon} + p_r^2 ) \mathrm{d} \mu_\sigma \le \int_{\ell^2} \exp( -p_2^{2+\varepsilon} + p_2^2 ) \mathrm{d} \mu_\sigma <\infty.\]
\end{example}
Our main results will prove that we can exchange the rolls of $p_2$ and $p_r$ in above example and the integral will remain finite, cf. \cref{ex2}.
\section{Absolutely \texorpdfstring{$p$}{p}-Summing Operators}\label{sec:psum}
We apply the theory of absolutely $p$-summing operators, which were introduced by Pietsch
\cite{src:Pietsch:AbsSummingHilbertSchmidtAndPrenuclear,src:Pietsch:pSummingMeasureCharacterisation}.
For a textbook introduction to the subject, we refer to \cite{DiestelJarchowTonge.1995}.
Let $T : E \to F$ be a linear operator between normed spaces $E$ and $F$.
Then $T$ is \textbf{absolutely $\mathbf{p}$-summing} 
for $p \ge 1$ if {there exists $C > 0$ such that} for all finite selections $x_1, ..., x_k \in E$ we have
\begin{equation}
	\left( \sum_{n = 1}^k \left \Vert T x_n \right \Vert_F^p \right)^{1/p}
	\le
	C \sup_{\left \Vert \phi \right \Vert_{{E^*}} \le 1}
	\left( \sum_{n = 1}^k \left \vert {\phi} \left(  x_n \right) \right \vert^p \right)^{1/p} \, .
\end{equation}
Recall that a bounded linear operator $L:H\to H$ on a separable Hilbert space $H$ is called
\textbf{Hilbert-Schmidt} if there is some orthonormal basis $(e_n)$ of $H$ such that
\begin{equation}
	\sum_{n = 1}^{\dim H} \left \Vert L e_n \right \Vert_H^2 < \infty \, .
\end{equation}
More generally,
we call a bounded linear operator $L : H \to J$ between two separable Hilbert spaces \textbf{Hilbert-Schmidt} if $\vert L \vert = \sqrt{L^* L}$ is Hilbert-Schmidt.
\begin{remark}
	It is well-known that any Hilbert-Schmidt operator $L : H \to J$ is compact and in fact absolutely $p$-summing for all $p \ge 1$ \cite{src:Peczyski:HilbertSchmidtPSumming}.
	Because $\vert L \vert$ is also self-adjoint it can be diagonalised with some orthonormal basis $(e_n)$ of $H$ and a corresponding (possibly finite) square-summable sequence $( \lambda_n )$ of real eigenvalues.
	By the polar decomposition, $L = U \circ \vert L \vert$ for some partial isometry $U : H \to J$, there is also an orthonormal set $(f_n )$ in $J$ such that
	\begin{equation}\label{eq:HSdecomp}
		L = \sum_{n = 1}^{\dim H} \lambda_n \left \langle e_n, \cdot \right \rangle_H f_n \, .
	\end{equation}
\end{remark}
An important result of Pietsch is a factorisation theorem. We only state a simple corollary of it here. 
\begin{theorem}
	\label{cor:GoodFactorisation}
	Let $L : H \to B$ be an absolutely $2$-summing
	map from a separable Hilbert space $H$ to a Banach space $B$.
	Then $L$ factorises through another separable Hilbert space $J$ and some linear maps $j : H \to J$ and $k : J \to B$ such that $j$ is Hilbert--Schmidt and $k$ is compact.
\end{theorem}
\begin{proof}
	It is the statement of Pietsch's factorisation theorem that $L$ can be factorised through an absolutely 2-summing operator $j'$ from $H$ to a separable Hilbert space $J$ and a bounded operator $k':J\to B$, cf. \cite[Theorem 2.13]{DiestelJarchowTonge.1995}.
	This implies the Hilbert--Schmidt property of $j'$, cf. \cite[\textsection 6 Prop. I.1]{LiQueffelec.2017}.
	\begin{equation}
		\sum_{n = 1}^\infty \delta_n^2 \lambda_n^2 < \infty
		\qquad \text{and} \qquad
		\lim_{n \to \infty} \delta_n^2 = \infty \, .
	\end{equation}
	Then the map $j : H \to J, e_n \to \delta_n \lambda_n f_n$ extended by linearity is Hilbert--Schmidt, while $l : J \to J, f_n \mapsto f_n / \delta_n$ linearly extends to a compact operator.
	Moreover, $j' = l \circ j$ by construction.
	Hence, $k = k' \circ l$ is compact and $L = k \circ j$.
\end{proof}
\section{The Integrability}\label{sec:integrability}
In this \lcnamecref{sec:integrability}, we prove \cref{mainthm}. 
It establishes the finiteness of the integral over the square of a continuous seminorm if a suitable submultiplicative counterterm is added.
We recall that a function $f:[0,\infty)\to[0,\infty)$ is called submultiplicative if
\begin{equation} f(xy) \le f(x)f(y) \qquad\mbox{for all}\ x,y\in[0,\infty).  \end{equation}
{Simple} examples are {given by} the functions $f(x)=|x|^a$ for $a\ge 0$ or $f(x)=\ln(a+x)$ for $a\ge e$ {and we refer to} \cite{GustavssonMaligrandaPeetre.1989} for more details and examples.
We also remark that products of submultiplicative functions are submultiplicative, which easily follows from the definition.

We will require the submultiplicative function to grow faster than quadratically in an appropriate sense. This is reflected in the limit condition
\begin{equation}\label{eq:superquad}
	f(x)>0\ \mbox{for all}\ x>0 \quad\mbox{and}\quad \lim\limits_{x\to 0}f(x)/x^2 = 0\,.
\end{equation}
Following up on above examples, these assumptions are satisfied if $f(x)=x^a$ for $a>2$ or $f(x)=x^2\ln(a+x)$ for $a \ge e$. 
We also remark that while the limit assumption \cref{eq:superquad} is taken at zero, it also characterizes the limit $x \to \infty$ since submultiplicativity implies $f(x) \ge f(1) / f(1/x)$.

The following \lcnamecref{lem:supquad} elucidates the usefulness of the assumption \cref{eq:superquad}.
\begin{lemma}\label{lem:supquad}
	Assume $f: [0,\infty) \to [0,\infty)$ is continuous, submultiplicative and satisfies \cref{eq:superquad}.
	If $(x_n)_{n \in \mathbb N} \subseteq [0,\infty)$ is a sequence satisfying $\lim\limits_{n \to \infty} f(x_n) = 0$, then $\lim\limits_{n \to \infty} x_n = 0$.
\end{lemma}
\begin{proof}
	We use proof by contradiction, so w.l.o.g. (otherwise restrict to a subsequence) assume there is $\bar x \in (0,\infty]$ with $\lim_{n \to \infty} x_n = \bar x$. If $\bar x$ is finite, then the continuity of $f$ implies $f(\bar x)  = 0$ contradicting \cref{eq:superquad}. Otherwise, by the submultiplicativity of $f$,
	\begin{equation}  f(x_n) \ge f(1)/f(1/x_n) \xrightarrow{n\to\infty} \infty,  \end{equation}
	{contradicting} the assumptions.
\end{proof}
The preceding \lcnamecref{lem:supquad} enables the following dominance result.
\begin{lemma}\label{lem:exponentboundsubmult}
	Let $X$ and $Y$ be normed spaces, $Z$ a reflexive Banach space and let $\kappa_X : Z \to X$ and $\kappa_Y : Z \to Y$ be compact linear maps such that $\kappa_X z = 0$ implies $\kappa_Y z = 0$ for all $z \in Z$.
	Furthermore, assume $f:[0,\infty)\to [0,\infty)$ is continuous, submultiplicative, and satisfies \cref{eq:superquad}.
	Then
	\begin{equation}
		\sup_{z \in Z} \left[ - f(\left \Vert \kappa_X z \right \Vert_X) + \left \Vert \kappa_Y z \right \Vert_Y^2 - \frac{1}{2} \left \Vert z \right \Vert_Z^2 \right] < \infty\,.
	\end{equation}
\end{lemma}
\begin{proof}
	Define $F: Z \to \mathbb R$ as
	\begin{equation} F(z) = - f(\left \Vert \kappa_X z \right \Vert_X) + \left \Vert \kappa_Y z \right \Vert_Y^2 - \frac{1}{2} \left \Vert z \right \Vert_Z^2.  \end{equation}
	We want to prove that under the given assumptions $F$ is bounded from above.
	{Arguing by contradiction, there is} a sequence $(z_n)_{n\in\mathbb N}$ such that $\lim_{n \to \infty}F(z_n) = \infty$.
	W.l.o.g., we can assume $\| \kappa_Y z_n \|_Y >0$ for all $n \in \mathbb N$.
	Since {$\liminf_{ n \to \infty} F(z_n) / \|\kappa_Y z_n\|^2 \ge 0$}, both {sequences}
	\begin{equation}
		f(\| \kappa_X z_n \|_X) / \| \kappa_Y z_n \|_Y^2
		\qquad \text{and} \qquad
		\| z_n \|_Z^2 / \| \kappa_Y z_n \|_Y^2
	\end{equation}
	must be bounded.
	The latter implies that $z_n/\|\kappa_Y z_n\|_Y$ has a weakly convergent subsequence, to which we restrict from now on.
	Let $\bar z$ be the weak limit of {$(z_n/\|\kappa_Y z_n\|_Y)$}.
	Since $f$ is submultiplicative, we find
	\begin{equation}
		0
		\le
		f \left( \frac{\|\kappa_X z_n\|_X}{\|\kappa_Y z_n\|_Y} \right)
		\le
		\frac{f(\|\kappa_X z_n\|_X)}{\|\kappa_Y z_n\|_Y^2} \cdot \|\kappa_Y z_n\|_Y^2
		f(1/\|\kappa_Y z_n\|_Y)
		\xrightarrow{n\to \infty} 0.
	\end{equation}
	{where the limit follows from \cref{eq:superquad}.}
	Using \cref{lem:supquad}, this implies
	\begin{equation} 0
		=
		\lim_{n \to \infty} \frac{\left \Vert \kappa_X z_{n} \right \Vert_X}{\left \Vert \kappa_Y z_{n} \right \Vert_Y}
		=
		\lim_{n \to \infty} \left \Vert \kappa_X \frac{z_{n}}{\left \Vert \kappa_Y z_{n} \right \Vert_Y} \right \Vert_X
		=
		\left \Vert \kappa_X \bar{z} \right \Vert_X
		\, ,
	\end{equation}
	by the compactness of $\kappa_X$. 
	{In turn, this} implies $\kappa_Y \bar{z} = 0$ by assumption.
	However, {then}
	\begin{equation}
		1
		=
		\lim_{n \to \infty} \frac{\left \Vert \kappa_Y z_{n} \right \Vert_Y}{\left \Vert \kappa_Y z_{n} \right \Vert_Y}
		=
		\lim_{n \to \infty} \left \Vert \kappa_Y \frac{z_{n}}{\left \Vert \kappa_Y z_{n} \right \Vert_Y} \right \Vert_Y
		\neq
		\left \Vert \kappa_Y \bar{z} \right \Vert_Y
		=
		0 ,
	\end{equation}
	contradicting the compactness of $\kappa_Y$.
\end{proof}
We shall derive the sought integrability by approximating the integral on finite-dimensional subspaces.
The justification comes from \cite{src:Bogachev:GaussianMeasures}.
\begin{theorem}
	\label{thm:FiniteDimensionalApproximation}
	{Let $\mu$ be a Radon Gaussian measure on a locally convex space $X$,} $( e_n )$ an orthonormal basis of the Cameron--Martin space $H(\mu)$ and set
	\begin{equation}\label{def:Pn}	P_n x = \sum_{k = 1}^n R_\mu^{-1}(e_k)(x) \, e_k \qquad\mbox{for all}\ n \in \mathbb{N}_{\le \dim H(\mu)}\ \mbox{and}\ x\in X.
	\end{equation}
	Then all $P_n$ are $(\mathcal{B}(X)_\mu,\mathcal{B}(X))$-measurable and all $\mu \circ P_n^{-1}$ are Radon Gaussian measures.
	Furthermore, if $\dim H(\mu) < \infty$, then $P_{\dim H(\mu)} = \mathrm{id}_X$ $\mu$-almost everywhere.
	Otherwise, $P_n \xrightarrow{n\to\infty} \mathrm{id}_X$ pointwise $\mu$-almost everywhere and $\mu \circ P_n^{-1}$ converges weakly to $\mu$.
\end{theorem}
\begin{proof}
	The measurability and the proof that all $\mu \circ P_n^{-1}$ are Radon Gaussian measures can be extracted from \cite[Theorem 3.7.6]{src:Bogachev:GaussianMeasures}.
	The weak convergence in the infinite dimensional case is treated in \cite[Proposition 3.8.12]{src:Bogachev:GaussianMeasures}.
	In the finite-dimensional case $H(\mu)$ is closed in $X$ such that $\mu$ is supported on $H(\mu)$ by \cite[Theorem 3.6.1]{src:Bogachev:GaussianMeasures}.
\end{proof}
It is instructive to calculate the corresponding characteristic functions in the above case.
For all $\phi\in X^*$, this yields
\begin{equation}\label{eq:charfunct}
	\begin{aligned}
		&\int_X \exp \left[ i \phi \left( x \right) \right] \mathrm{d} \left( \mu \circ P_n^{-1} \right)\\
		&=
		\int_X \exp \left[ i \sum_{k = 1}^n \phi \left( e_k \right) R_\mu^{-1} \left( e_k \right) \left( x \right) \right] \mathrm{d} \mu \left( x \right)
		=
		\exp \left[ - \frac{1}{2} \sum_{k = 1}^n \phi \left( e_k \right)^2\right] \\
		&=
		\left( 2 \pi \right)^{-n/2} \int_{\mathbb{R}^n} \exp \left[ i \phi \left( \sum_{k = 1}^n x_k e_k \right) - \frac{1}{2} \sum_{k = 1}^n x_k^2 \right] \mathrm{d}^n x \, ,
	\end{aligned}
\end{equation}
by \cref{thm:CameronMartin}.
Since Radon Gaussian measures on locally convex spaces are uniquely determined by their characteristic functions \cite[Theorem 2.2.4, Proposition A.3.13, Proposition A.3.18]{src:Bogachev:GaussianMeasures}, it is clear that
\begin{equation}
	\label{eq:FiniteDimMeasureEquality}
	\mu \circ P_n^{-1} = \nu_n \circ \Lambda_n^{-1} \circ \iota_\mu^{-1}
\end{equation}
where $\nu_n$ is the normalised standard Gaussian measure on $\mathbb{R}^n$, $\iota_\mu : H(\mu) \to X$ is the inclusion map and
\begin{equation}\label{def:Lambdan}
	\Lambda_n : \mathbb{R}^n \to H(\mu),\ (x_1,\ldots,x_n) \mapsto \sum_{k = 1}^n x_k e_k.
\end{equation}
Building on the above observations, we can now state and prove the main result of this \lcnamecref{sec:integrability}.
\begin{theorem}\label{mainthm}
	Let $\mu$ be a centred Radon Gaussian measure on a locally convex space $X$ 
	and let $p$ and $q$ be continuous seminorms on $X$. 
	Assume that the {natural map} $\iota_\mu^{p+q}: H(\mu) \to H(\mu)_{p+q}$ is absolutely 2-summing 
	and that the natural map $H(\mu)_{p+q} \to H(\mu)_p$ is injective.
	Then, for any continuous and submultiplicative function $f: [0,\infty) \to [0,\infty)$ satisfying
	\cref{eq:superquad},
	we have
	\begin{equation}\label{eq:mainthm}
		\int_X \exp \left[ -f(p(x)) + q(x)^2 \right] \mathrm d \mu(x) < \infty\,. 
	\end{equation}
\end{theorem}
\begin{proof}	
	According to \cref{cor:GoodFactorisation}, there is a separable Hilbert space $J$, a Hilbert-Schmidt map $j : H(\mu) \to J$ and a compact linear map $k : J \to H(\mu)_{p+q}$ such that $\iota_\mu^{p+q} = k \circ j$.
	Throughout this proof, we fix an orthonormal basis $(e_n)$ of $H(\mu)$, an orthonormal set $(f_n) \subseteq J$ and a sequence $(\lambda_n)$ in $\mathbb R$ such that
	\begin{equation}
		j = \sum_{n = 1}^{\dim H(\mu)} \lambda_n \left \langle e_n, \cdot \right \rangle_{H(\mu)} f_n \, ,
	\end{equation}
	cf. \cref{eq:HSdecomp}.
	Note that w.l.o.g. (otherwise, rescale $k$ on a finite-dimensional subspace of $J$), we can assume $|\lambda_n|^2<\frac 12$, since the only accumulation point of the set $\{\lambda_n\}$ can be zero.
	We further denote $P_n$ and $\Lambda_n$ as in \cref{thm:FiniteDimensionalApproximation,def:Lambdan}, respectively.
	We consider the following natural maps:
	\begin{equation}
		\pi^p_{p+q} : H(\mu)_{p+q} \to H(\mu)_{p}
		\qquad \text{and} \qquad
		\pi^q_{p+q}: H(\mu)_{p+q} \to H(\mu)_{q} \, .
	\end{equation}
	Since $\pi^p_{p+q}$ is injective by assumption, we have that $\pi^p_{p+q} x = 0$ implies $p(x) + q(x) = 0$ such that $\pi^q_{p+q} x = 0$.
	{We can now apply} \cref{lem:exponentboundsubmult} with $X = H(\mu)_p$, $\kappa_X = \pi^p_{p+q} \circ k$, $Y = H(\mu)_q$, $\kappa_Y = \pi^q_{p+q} \circ k$ and $Z = J$.
	This yields
	\begin{equation}
		\label{eq:NormBalancingBoundsubmult}
		C
		=
		\sup_{x \in J} \left[ -f(p \left( k x \right)) + q \left( k x \right)^2 - \frac{1}{2} \left \Vert x \right \Vert_J^2 \right]
		< 
		\infty .
	\end{equation}
	For all $n\le \dim H(\mu)$ and $x\in\mathbb R^n$, we now find
	\begin{equation}\label{eq:IntegrandBoundAndRemainder}
		\begin{aligned}
			&-f\left(p \left( \iota_\mu^{p+q} \Lambda_n x \right)\right)   +  q \left( \iota_\mu^{p+q} \Lambda_n x \right)^2 - \frac{1}{2} \sum_{a = 1}^n x_a^2 \\
			&=
			- \frac{1}{2} \sum_{a = 1}^n \left( 1 - \left \Vert j e_a \right \Vert^2 \right) x_a^2 - f\left(p \left( k j \Lambda_n x \right)\right) + q \left( k j \Lambda_n x \right)^2 - \frac{1}{2} \sum_{a = 1}^n \left \Vert j e_a \right \Vert_J^2 x_a^2 \\
			&\le
			- \frac{1}{2} \sum_{a = 1}^n \left( 1 - \left \vert \lambda_a \right \vert^2 \right) x_a^2 +
			\sup_{y \in \mathbb{R}^n} \left[ -f\left(p \left( k j \Lambda_n y \right)\right) +  q \left( k j \Lambda_n y \right)^2 - \frac{1}{2} \left \Vert j \Lambda_n y \right \Vert_J^2 \right] \\
			&\le
			- \frac{1}{2} \sum_{a = 1}^n \left( 1 - \left \vert \lambda_a \right \vert^2 \right) x_a^2 + C \, .
		\end{aligned}
	\end{equation}
	Using \cref{thm:FiniteDimensionalApproximation,lem:weakconvunbound,eq:FiniteDimMeasureEquality}, we then obtain
	\begin{equation*}
		\begin{aligned}
			&\int_X  \exp  \left[ -f\left(p \left( x \right)\right) + q \left( x \right)^2 \right] \mathrm{d} \mu \left( x \right) \\
			&\le
			\liminf_{n \to \dim H(\mu)} \int_{{X}} \exp \left[ -f\left(p \left( x \right)\right) +  q \left( x \right)^2 \right] \mathrm{d} (\mu\circ P_n^{-1}) \left( x \right) \\
			&{=}
			\liminf_{n \to \dim H(\mu)} \left( 2 \pi \right)^{-n/2} \\&\qquad \times 
			\int_{\mathbb{R}^n} \exp \left[
			- f\left(p \left( \iota_\mu^{p+q} \Lambda_n x \right)\right)
			+  q \left( \iota_\mu^{p+q} \Lambda_n x \right)^2
			- \frac{1}{2} \sum_{a = 1}^n x_a^2
			\right] \mathrm{d}^n x \\
			&\le
			e^C \liminf_{n \to \dim H(\mu)} \left( 2 \pi \right)^{-n/2}
			\int_{\mathbb{R}^n} \exp \left[
			- \frac{1}{2} \sum_{a = 1}^n \left( 1 - \left \vert \lambda_a \right \vert^2 \right) x_a^2
			\right] \mathrm{d}^n x \\
			&=
			e^C \liminf_{n \to \dim H(\mu)} \prod_{a = 1}^n \frac{1}{\sqrt{1 - \left \vert \lambda_a \right \vert^2}}
		\end{aligned}
	\end{equation*}
	However, since $\left \vert \lambda_a \right \vert^2 \le 1/2$ and $- \ln \sqrt{1 - t} < t$ for all $t \in \left[0, 1/2 \right]$,
	\begin{equation}\label{eq:HSestimate}
		\begin{aligned}
			\lim_{n \to \dim H(\mu)} \prod_{a = 1}^n \frac{1}{\sqrt{1 - \left \vert \lambda_a \right \vert^2}}
			&=
			\exp \left[ - \sum_{a = 1}^{\dim H(\mu)} \ln \sqrt{1 - \left \vert \lambda_a \right \vert^2} \right]
			\\&\le
			\exp \left[ \sum_{a = 1}^{\dim H(\mu)} \left \vert \lambda_a \right \vert^2 \right] \, ,
		\end{aligned}
	\end{equation}
	which is finite because $j$ is Hilbert-Schmidt.
\end{proof}
\begin{remark}\label{rem:Lr}
	The injectivity assumption in above \lcnamecref{thm:FiniteDimensionalApproximation} is automatically satisfied in the case that $p$ and $q$ are $L^r$-norms (w.r.t. two equivalent measures), by the existence of almost everywhere convergent subsequences.
\end{remark}
On nuclear spaces, we obtain the following simple \lcnamecref{cor:NuclearSpaceIntegrability}.
\begin{corollary}
	\label{cor:NuclearSpaceIntegrability}
	Let $\mu$ be a centred Radon Gaussian probability measure on a nuclear space $X$ and let $p$ and $q$ be continuous seminorms on $X$ with the property that the natural map $\iota_{p+q}^p : X_{p+q} \to X_p$ is injective.
	Then, for all $\alpha, \varepsilon > 0$, we have
	\begin{equation}
		\exp \left[ - p^{2+\varepsilon} + \alpha q^2 \right] \in L^1(\mu) \, .
	\end{equation}
\end{corollary}
\begin{proof}
	W.l.o.g. we fix $\alpha=1$. Otherwise, we can replace $q$ by $\sqrt \alpha q$.
	By the nuclearity of $X$, there is a continuous Hilbert norm $r$ such that $p + q \le r$ and the natural map $\iota_r^{p+q} : X_r \to X_{p+q}$ is nuclear.
	By continuity of the inclusion map $\iota_\mu : H(\mu) \to X$ and the natural map $\iota^r : X \to X_r$,
	\begin{equation}
		\iota_r^{p+q} \circ \iota^r \circ \iota_\mu : H(\mu) \to X_{p+q}
		\qquad \text{is also nuclear.}
	\end{equation}
	Moreover, $H(\mu)_{p+q}$ carries the subspace topology of $X_{p+q}$ such that {the natural map} $\pi_\mu^{p+q} : H(\mu) \to H(\mu)_{p+q}$ is nuclear as well and hence absolutely 2-summing \cite[Satz 5]{src:Pietsch:AbsSummingHilbertSchmidtAndPrenuclear}.
	That $\pi_{p+q}^p : H(\mu)_{p+q} \to H(\mu)_p$ is injective follows from the injectivity of $\iota_{p+q}^p$.
	Hence, since $f(x)=x^{2+\varepsilon}$ for $x\ge 0$ is submultiplicative and satisfies \cref{eq:superquad}, all assumptions of \cref{mainthm} are satisfied and the statement follows.
\end{proof}
\begin{remark}\label{rem:intro}
	Combining \cref{rem:Lr,cor:NuclearSpaceIntegrability}, the introductory example in \cref{eq:IntegralWithCounterterm} is finite.
\end{remark}
\section{The Uniform Integrability}
\label{sec:UniformIntegrability}
For applications it is useful not only to know the finiteness of an expression, but to have an actual approximating sequence.
In this \lcnamecref{sec:UniformIntegrability}, we present a sufficient condition for the finite dimensional approximations used as upper bound in the previous \lcnamecref{sec:integrability} to converge towards the actual value of the integral.
We only treat the case that the Cameron--Martin space $H(\mu)$ is infinite-dimensional, because otherwise the statement is trivial by virtue of \cref{thm:FiniteDimensionalApproximation}.

The key towards a proof of convergence are two uniformity statements:
Uniform tightness of the measures $(\mu \circ P_n^{-1})$ and uniform integrability of the integrand \cite[Lemma 3.8.7]{src:Bogachev:GaussianMeasures}.
As is well-known, the weak convergence of a sequence of Radon measures on a metric space to another Radon measure implies their uniform tightness \cite[Theorem 8.6.4]{src:Bogachev:MeasureTheory2}.
Consequently, in such a setting, we just need to prove the uniform integrability.
We summarise this in the following \namecref{lem:UnboundedMeasurableConvergence}.
\begin{lemma}[{\cite[Theorem 8.6.4]{src:Bogachev:MeasureTheory2},\cite[Lemma 3.8.7]{src:Bogachev:GaussianMeasures}}]
	\label{lem:UnboundedMeasurableConvergence}
	Let $X$ be a metric space, $\left( \mu_n \right)_{n \in \mathbb{N}}$ a sequence of Radon measures weakly converging to a Radon measure $\mu$ and $f \in L^1(\mu)$ be continuous.
	Suppose that
	\begin{equation}
		\lim_{R \to \infty} \sup_{n \in \mathbb{N}} \int_{\left \vert f \right \vert \ge R} \left \vert f \right \vert \mathrm{d} \mu_n = 0 \, .
	\end{equation}
	Then
	\begin{equation}
		\lim_{n \to \infty} \int_X f \, \mathrm{d} \mu_n = \int_X f \, \mathrm{d} \mu \, .
	\end{equation}
\end{lemma}
Using this well-known statement, we obtain the following \lcnamecref{thm:FiniteDimensionalConvergence}.
\begin{theorem}
	\label{thm:FiniteDimensionalConvergence}
	{
		Let $\mu$ be a centred Radon Gaussian measure on a metrizable locally convex space $X$ with infinite-dimensional Cameron--Martin space $H(\mu)$ and let $p$ and $q$ be continuous seminorms on $X$. 
		Assume that the natural map $\iota_\mu^{p+q}: H(\mu) \to H(\mu)_{p+q}$ is absolutely 2-summing and that the natural map $H(\mu)_{p+q} \to H(\mu)_p$ is injective.
		Furthermore, suppose that some factorisation $\iota_\mu^{p+q} = k \circ j$ as in \cref{cor:GoodFactorisation} has the property that
		\begin{equation}
			\sum_{n = 1}^\infty \lambda_n^2 \ln \left( n + 1 \right) < \infty \, ,
		\end{equation}
		where $(\lambda_n)$ denotes the sequence of eigenvalues of $|j|$ sorted in descending order.
		Finally, let $(e_n)$ be a corresponding eigenbasis of $H(\mu)$ diagonalizing $|j|$ and construct $P_n$ as in \cref{thm:FiniteDimensionalApproximation}.
		
		Then, for any continuous and submultiplicative function $f: [0,\infty) \to [0,\infty)$ satisfying \cref{eq:superquad},
		we have
		\begin{equation}
			\begin{aligned}
				\label{eq:integralConvergenceOnLocallyConvexX}
				\lim_{n \to \infty} \int_X & \exp \left[ -f\left(p \left( x \right)\right) + \alpha q \left( x \right)^2 \right] \mathrm{d} \left( \mu \circ P_n^{-1} \right) \left( x \right) \\
				&=
				\int_X \exp \left[ -f\left(p \left( x \right)\right) + \alpha q \left( x \right)^2 \right] \mathrm{d} \mu \left( x \right)
				< \infty
				\quad\mbox{for all}\ \alpha>0
				.
			\end{aligned}
		\end{equation}
	}
\end{theorem}
\begin{proof}
	As in the proof of \cref{mainthm}, we will w.l.o.g. assume that $\lambda_n^2< 1/2$ for all $n \in \mathbb N$.
	
	Fix $\alpha>0$ and let $C$ be as defined in \cref{eq:NormBalancingBoundsubmult} with $q$ replaced by $\sqrt \alpha q$. Then, for any $R>0$, it easily follows that {for all $x\in H(\mu)$}
	\begin{equation}
		-f(p(\iota_\mu^{p+q}x)) + \alpha q(\iota_\mu^{p+q}x)^2
		\ge
		\frac{1}{2} (R+C) \qquad \text{implies} \qquad \|jx\|_J^2 \ge R \, .
	\end{equation}
	Hence, defining $\Lambda_n$ as in \cref{def:Lambdan} and mimicking \cref{eq:IntegrandBoundAndRemainder} we find
	\begin{equation}
		\begin{aligned}
			&\int_{-f \circ p + \alpha q^2 \ge \frac{1}{2} (R+C)}  \exp \left[
			-f(p(x))+\alpha q(x)^2 \right] \mathrm d (\mu \circ P_n^{-1})(x) \\
			& \le (2\pi)^{-n/2} \int_{\|j \Lambda_n x \|_J^2\ge R} \exp \left[
			-f(p(\Lambda_n x))+\alpha q(\Lambda_n x)^2 - \frac{1}{2} \sum_{k=1}^{n} x_k^2 \right] \mathrm d^nx\\
			& \le (2\pi)^{-n/2} e^C \int_{\sum_{k = 1}^n \lambda_k^2x_k^2 \ge R}
			\exp \left[ -\frac{1}{2} \sum_{k=1}^{n}(1-\lambda_k^2)x_k^2 \right] \mathrm d^nx.
		\end{aligned}	
	\end{equation}
	By \cref{lem:UnboundedMeasurableConvergence}, the statement follows if we can prove that the right hand side converges to zero uniformly in $n \in \mathbb N$ as $R \to \infty$.
	For notational convenience, we define $I_n : \mathbb{R} \to \mathbb{R}$ as
	\begin{equation}
		I_n(R) =  \left( 2 \pi \right)^{-n/2}
		\int_{\sum_{a = 1}^n {\lambda_a^2 x_a^2} \ge R}
		\exp \left[ - \frac{1}{2} \sum_{a = 1}^n { \left( 1 - \lambda_a^2 \right) x_a^2} \right]
		\mathrm{d}^n {x}
	\end{equation}
	and set
	\begin{equation}\label{def:cn}
		c_n = 1 + {\lambda_n}^{-2} \left[ \ln \left( n + 1 \right) \right]^{-1} \qquad\mbox{for}\ n\in\mathbb N.
	\end{equation}
	%
	For all $n\ge 2$ and $R\ge 0$, we then obtain the recursion formula
	\begin{equation}\label{eq:recursion}
		\begin{aligned}
			I_n \left( R \right)
			=\,&
			2 \left( 2 \pi \right)^{-1/2} \int \limits_{\sqrt{R/ ( c_n {\lambda_n^2} )}}^\infty
			\exp \left[ - \frac{1}{2} \left( 1 - {\lambda_n^2} \right) {x^2} \right]
			I_{n-1} \left( R - {\lambda_n^2} {x^2} \right)
			\mathrm{d} {x} \\
			&
			+ 2 \left( 2 \pi \right)^{-1/2} \int \limits_0^{\sqrt{R/ ( c_n {\lambda_n^2} )}}
			\exp \left[ - \frac{1}{2} \left( 1 - {\lambda_n^2} \right) {x^2} \right]
			I_{n-1} \left( R - {\lambda_n^2} {x^2} \right)
			\mathrm{d} {x} .
		\end{aligned}
	\end{equation}
	%
	By the monotonicity of $I_n$, we have the bounds
	\begin{equation}
		I_{n-1}(R-\lambda_n^2 x^2 ) \le \begin{cases} \displaystyle I_{n-1}(0) = \prod_{a = 1}^{n-1}\frac{1}{\sqrt{1-\lambda_a^2}} & \mbox{for}\ x\in[\sqrt{R/(c_n\lambda_n^2)},\infty), \\ \displaystyle  I_{n-1}\left(\frac{c_n-1}{c_n}R\right) & \mbox{for}\ x\in[0,\sqrt{R/(c_n\lambda_n^2)}]. \end{cases}
	\end{equation}
	Inserting these into \cref{eq:recursion} and using the standard definitions
	\begin{equation} \Erf(x) = \frac 2{\sqrt \pi}\int_0^x e^{-t^2}\mathrm d t \qquad\mbox{and}\qquad \Erfc(x) = 1-\Erf(x) \qquad\mbox{for}\ x\in[0,\infty),  \end{equation}
	we obtain the recursive estimate
	\begin{equation}
		\begin{aligned}
			I_n \left( R \right)
			&\le
			\frac{ \Erfc \left( \sqrt{ R \frac{1 - {\lambda_n^2}}{2 {\lambda_n^2} c_n} } \right) }{\sqrt{1 - {\lambda_n^2}}}
			I_{n-1} \left( 0 \right)
			+
			\frac{ \Erf \left( \sqrt{ R \frac{1 - {\lambda_n^2}}{2 {\lambda_n^2} c_n} } \right) }{\sqrt{1 - {\lambda_n^2}}}
			I_{n-1} \left( \frac{c_n - 1}{c_n} R \right) \\
			&\le
			\left( \prod_{a = 1}^{n} \frac{1}{\sqrt{1-{\lambda_a^2}}} \right)
			\Erfc \left( \sqrt{ R \frac{1 - {\lambda_n^2}}{2 {\lambda_n^2} c_n} } \right) 
			+
			\frac{ I_{n-1} \left( \frac{c_n - 1}{c_n} R \right) }{\sqrt{1 - {\lambda_n^2}}} \, ,
		\end{aligned}
	\end{equation}
	where we used the estimate $\mathrm{Erf} \le 1$ in the second inequality.
	Iteratively applying this upper bound, we find
	\begin{equation}\label{eq:boundafteriteration}
		\begin{aligned}
			I_n \left( R \right)
			&\le
			\left( \prod_{a = 1}^{n} \frac{1}{\sqrt{1-{\lambda_a^2}}} \right)
			\sum_{b = 0}^{n-1}
			\Erfc \left( \sqrt{ \frac{R}{2} \left( \prod_{m = 0}^{b-1} \frac{c_{n-m} - 1}{c_{n-m}} \right) \frac{ 1 - {\lambda_{n-b}^2}}{{\lambda_{n-b}^2} \, c_{n-b}}} \right) \\
			&=
			\left( \prod_{a = 1}^{n} \frac{1}{\sqrt{1-{\lambda_a^2}}} \right)
			\sum_{b = 1}^{n}
			\Erfc \left( \sqrt{ \frac{R}{2} \left( \prod_{m = b+1}^{n} \frac{c_{m} - 1}{c_{m}} \right) \frac{ 1 - {\lambda_b^2}}{{\lambda_b^2} \, c_b}} \right) \, .
		\end{aligned}
	\end{equation}
	Inserting the choice \cref{def:cn} and using $\ln(1+x)\le x$ for all $x>0$, we have
	\begin{equation}\label{ineq:firstterm}
		\begin{aligned}
			\prod_{m = b+1}^{n} \frac{c_{m} - 1}{c_{m}}
			&\ge
			\exp \left[ - \sum_{m = 1}^{\infty} \ln \left( 1 + {\lambda_m^2} \ln \left( m + 1 \right) \right) \right] \\
			&\ge
			\exp \left[ - \sum_{m = 1}^{\infty} {\lambda_m^2} \ln \left( m + 1 \right) \right]
			=: \varrho > 0.
		\end{aligned}
	\end{equation}
	Since $\lim_{b \to \infty} {\lambda_b^2} \ln \left( b + 1 \right) = 0$, by the assumptions, we can pick $\beta > 0$ such that
	\begin{equation}\label{ineq:secondterm}
		\frac{ 1 - {\lambda_b^2}}{{\lambda_b^2} c_b}
		=
		\frac{1 - {\lambda_b^2}}{1 + {\lambda_b^2} \ln \left( b + 1 \right)} \ln \left( b + 1 \right)
		\ge
		\beta \ln \left( b + 1 \right) \, .
	\end{equation}
	Now, inserting \cref{ineq:firstterm,ineq:secondterm} into \cref{eq:boundafteriteration} and using the fact that $\mathrm{Erfc} ( x ) < \exp [ - x^2 ]$ as well as the estimate \cref{eq:HSestimate}, we find
	\begin{equation}\label{eq:Riemannfinish}
		\begin{aligned}
			\sup_{n \in \mathbb{N}} I_n \left( R \right)
			&\le
			\left( \prod_{a = 1}^{\infty} \frac{1}{\sqrt{1-{\lambda_a^2}}} \right)
			\sum_{b = 1}^{\infty}
			\exp \left[ - \frac{\varrho \beta}{2} R \ln \left( b + 1 \right) \right] \\
			&\le
			\exp \left[ \sum_{a = 1}^\infty {\lambda_a^2} \right]
			\left[ \zeta \left( \frac{\varrho \beta}{2} R \right) - 1 \right]
		\end{aligned}
	\end{equation}
	whenever $R > 2/(\varrho \beta)$ with $\zeta(s) =\sum_{n=1}^{\infty}n^{-s}$ for $s>1$ being the Riemann zeta function.
	The right hand side of \cref{eq:Riemannfinish} tends to zero as $R\to\infty$ and hence the proof is complete.
\end{proof}
As before, a very useful \lcnamecref{cor:uniformint-nuclear} is obtained on metrizable nuclear spaces. 
\begin{corollary}\label{cor:uniformint-nuclear}
	Let $\mu$ be a centred Radon Gaussian probability measure on a metrizable nuclear space $X$ and let $p$, $q$ be continuous seminorms on $X$ with the property that the natural map $\iota_{p+q}^p : X_{p+q} \to X_p$ is injective.
	Then there is some orthonormal basis $( e_n )_{n \in \mathbb{N}}$ of $H(\mu)$ such that
	\begin{equation}
		\lim_{n \to \dim H(\mu)}
		\int_X \exp \left[ -p^{2 + \varepsilon} + \alpha q^2 \right] \mathrm{d} \left( \mu \circ P_n \right) \\
		=
		\int_X \exp \left[ -p ^{2 + \varepsilon} + \alpha q^2 \right] \mathrm{d} \mu
		< \infty
	\end{equation}
	for all $\alpha, \varepsilon > 0$,
	where $P_n$ is defined as in \cref{thm:FiniteDimensionalApproximation}.
\end{corollary}
\begin{proof}
	By {the} nuclearity of $X$ there is a continuous Hilbert seminorm $r > p + q$ on $X$ such that the natural map $\iota_r^{p+q} : X_r \to X_{p+q}$ is nuclear and in particular compact.
	{Likewise, there is another continuous Hilbert seminorm $s > r$ such that the natural map $\iota_s^r : X_s \to X_r$ is nuclear as well.
		Since, $\iota_\mu^r : H(\mu) \to X_r$ factorises continuously over $\iota_s^r$, $\iota_\mu^r$ is a nuclear map between Hilbert spaces making the eigenvalues of $\left \vert \iota_\mu^r \right \vert$ summable.
		Thus, $\iota_\mu^{p+q} = \iota_r^{p+q} \circ \iota_\mu^r$ provides a required factorisation for Theorem \ref{thm:FiniteDimensionalConvergence}.} Combining this with the fact that $f(x)=x^{2+\varepsilon}$ is submultiplicative and satisfies \cref{eq:superquad}, proves the statement.
\end{proof}
\begin{remark}
	Similar to \cref{rem:intro},
	this result applies to the introductory example \cref{eq:IntegralWithCounterterm}.
\end{remark}
\section{The Case of Measurable Seminorms}\label{sec:MeasurableSeminorms}
In this section we will weaken the continuity requirement on the involved seminorms to mere measurability plus the condition that the closure of the Cameron--Martin space with respect to the sum of the seminorms has full measure.
This enables us to push forward the Gaussian measure via a measurable linear map to a metrizable space on which the seminorms are continuous.
Furthermore, we show that the Cameron--Martin space of the pushforward measure is similar enough to the original one such that the convergence part of \cref{thm:FiniteDimensionalConvergence} is maintained.

Since measurable seminorms are only meaningfully defined on subspaces of full measure, we introduce the notation
\begin{equation}
	\mathrm{cl}^{X_0}_r M
	\coloneqq
	\left \{
	x \in X_0 \, \Big| \, \exists \left( h_n \right)_{n \in \mathbb{N}} \subseteq M : \lim_{n \to \infty} r \left( x - h_n \right) = 0
	\right \}
\end{equation}
for the $r$-closure of a subset $M \subseteq X$ in the subspace $X_0\subseteq X$ and a seminorm $r:X_0\to\mathbb R$.

The first statement now gives us the measurability of the $r$-closure of the Cameron--Martin space.
\begin{lemma}
	Let $\mu$ be a centred Radon Gaussian probability measure on a locally convex space $X$
	and let $r$ be a $\mu$-mea\-sura\-ble seminorm on $X$.
	Then $\mathrm{cl}^{X_0}_r H(\mu) \in \mathcal{B}(X)_\mu$ for any $\mu$-measurable linear subspace $X_0 \subseteq X$ on which $q$ is a seminorm.
\end{lemma}
\begin{proof}
	By \cref{thm:CameronMartin}, for any $\varepsilon > 0$ and $h \in H(\mu)$,
	\begin{equation}
		B^r_\varepsilon \left( h \right) = \left\{ x \in X_0 : r \left( x - h \right) < \varepsilon \right\} \in \mathcal{B}(X)_\mu \, .
	\end{equation}
	Since $r$ is continuous on $H(\mu)$ by \cref{thm:MeasurableSeminormContinuousOnCameronMartin}, we obtain
	\begin{equation}
		\mathrm{cl}^{X_0}_r H(\mu)
		=
		\bigcap_{n \in \mathbb{N}} \bigcup_{h \in D}
		B^r_{1/n} \left( h \right)
	\end{equation}
	for any dense set $D \in H(\mu)$.
	Because $H(\mu)$ is separable the claim follows.
\end{proof}
We can now state our completion result.
Under the assumption that the closure of the Cam\-e\-ron--Martin space has full measure,
we prove that our integrals can be treated on a Banach space with continuous seminorms.
\begin{theorem}\label{thm:sepBSextension}
	Let $\mu$ be a centred Radon Gaussian measure on a locally convex space $X$ 
	and let $r$ be a $\mu$-mea\-sura\-ble seminorm on $X$. 
	Assume that for some linear subspace $X_0 \subseteq X$ of full $\mu$-measure on which $r$ is a seminorm, the $r$-closure $\mathrm{cl}^{X_0}_r H(\mu)$ has full $\mu$-measure.
	Then there is a separable Banach space $Y$ and a linear map $\pi : X \to Y$ such that
	\begin{itemize}
		\item $\pi$ is $(\mathcal{B}(X)_\mu,\mathcal{B}(Y)_{\bar\mu})$-measurable,
		\item $r(x) = \Vert \pi(x) \Vert_Y \eqqcolon \bar{r}(\pi x)$ for every $x \in X_0$,
		\item the pushforward measure $\bar{\mu} = \mu \circ \pi^{-1}$ is a centred Radon Gaussian probability measure on $Y$.
	\end{itemize}
	Additionally, we can choose $Y$ such that the following holds.
	\begin{enumerate}[(i)]
		\item Every $\mu$-measurable seminorm $p$ with $p \le r$ on $X_0$ induces a continuous seminorm $\bar{p}$ on $Y$ such that $p = \bar{p} \circ \pi$.
		If the natural map $\kappa_r^p : H(\mu)_r \to H(\mu)_p$ is injective, then the natural map $\kappa_{\bar{r}}^{\bar{p}} : H(\bar{\mu})_{\bar{r}} \to H(\bar{\mu})_{\bar{p}}$ is injective as well.
		\item For every orthonormal basis $(\bar{e}_n)$ of $H(\bar{\mu})$ there exists a unique orthonormal system $(e_n') \subseteq H(\mu)$ such that $\pi e_n'=\bar e_n$ for $n\in\mathbb N$. For any basis extension $(e_n)$ of $(e_n')$ and all $\bar{\mu}$-measurable functions $F : Y \to \mathbb{R}$,
		\begin{equation}
			\label{eq:FunctionONBIntegrationLimit}
			\lim_{n \to \dim H(\mu)} \int_X F \circ \pi \; \mathrm d (\mu \circ P_{n}^{-1})
			=
			\lim_{n \to \dim H(\bar{\mu})} \int_{Y} F \; \mathrm{d} (\bar{\mu} \circ \bar{P}_n^{-1})
		\end{equation}
		whenever the right-hand side exists.
		Here, $P_n$ and $\bar{P}_n$ are constructed as in \cref{thm:FiniteDimensionalApproximation} for the bases $(e_n)$ and $(\bar{e}_n)$, respectively.
		\item  Assume that the natural map $\iota_\mu^r : H(\mu) \to H(\mu)_r$ admits a factorisation $\iota_\mu^r = k \circ j$ in the sense of \cref{cor:GoodFactorisation} and for the decreasing sequence $(\lambda_n)$ of eigenvalues of $|j|$,
		\begin{equation} \sum_n \lambda_n^2 a_n < \infty \end{equation}
		for some sequence $(a_n)$ of positive real numbers.
		Then the natural map $\iota_{\bar{\mu}}^{\bar{r}} : H(\bar{\mu}) \to H(\bar{\mu})_{\bar{r}}$ admits a factorization $\iota_{\bar{\mu}}^{\bar{r}} = \bar{k} \circ \bar{j}$ in the sense of \cref{cor:GoodFactorisation} and for  the decreasing sequence $(\bar{\lambda}_n)$ of eigenvalues of $|\bar{j}|$,
		\begin{equation} \sum_n \bar{\lambda}_n^2 a_n < \infty \, . \end{equation}
		In particular, $\iota_{\bar{\mu}}^{\bar{r}}$ is absolutely $2$-summing if $\iota_\mu^r$ is.
	\end{enumerate}
\end{theorem}
\begin{proof}
	Define $Y$ as the $r$-completion of $\mathrm{cl}^{X_0}_r H(\mu)$ and $\pi: X \to Y$ as any linear extension of the corresponding natural map $\mathrm{cl}^{X_0}_r H(\mu) \to Y$.
	Then $Y$ is a separable Banach space and letting $D \subseteq H(\mu)$ denote some countable dense set in $H(\mu)$, we have that $\pi(D)$ is dense in $Y$.
	Consequently, any open set $U \subseteq Y$ can be written as
	\begin{equation}
		U = \bigcup_{h \in D} \left\{ z \in Y : \bar{r} \left( z - \pi h \right) < \varepsilon_h \right\}
	\end{equation}
	for some nonnegative numbers $(\varepsilon_h)_{h \in D}$.
	Hence,
	\begin{equation}
		\pi^{-1} \left( U \right) = \bigcup_{h \in D} B^{r}_{\varepsilon_h} \left( h \right) \in \mathcal{B}(X)_\mu,
	\end{equation}
	which proves that the map $\pi$ is $(\mathcal{B}(X)_\mu,\mathcal{B}(Y))$-measurable.
	It then trivially follows that $\pi$ is $(\mathcal{B}(X)_{\mu}, \mathcal{B}(Y)_{\bar{\mu}})$-measurable.
	Furthermore, $Y$ is Polish and thus $\bar{\mu}$ is Radon.
	To see that $\bar{\mu}$ is a centred Gaussian measure, note that for every $\phi \in Y^*$ the function $\phi \circ \pi$ is $\mu$-measurable.
	Hence, by \cite[Theorem 3.11.3]{src:Bogachev:GaussianMeasures}, $\phi \circ \pi \in X^*_\mu$ and
	\begin{equation}
		\hat{\bar{\mu}} \left( \phi \right)
		=
		\hat{\mu} \left( \phi \circ \pi \right)
		=
		\exp \left[ -\frac{1}{2} \left \Vert \phi \circ \pi \right \Vert^2_{X_\mu^*} \right] \, .
	\end{equation}
	This shows that $\bar{\mu}$ is Gaussian and also that $T : Y^* \to X^*_\mu, \phi \mapsto \phi \circ \pi$ extends to an isometry $T : Y^*_{\bar{\mu}} \to X^*_\mu$.
	Clearly, for any $L^2(\bar{\mu})$-convergent sequence $(\phi_n)_{n \in \mathbb{N}}$ in $Y^*$ the sequence $(\phi_n \circ \pi)_{n \in \mathbb{N}}$ is $L^2(\mu)$-convergent as well.
	Combined with the $(\mathcal{B}(X)_{\mu}, \mathcal{B}(Y)_{\bar{\mu}})$-measurability, $T \phi$ can be viewed as $\phi \circ \pi$ for any function $\phi \in Y^*_{\bar{\mu}}$ as well.
	
	Now we have verified the three primary claims and it remains to prove the assertions (i), (ii) and (iii).
	To that end, note that for all $x \in X_0$,
	\begin{equation}
		\begin{aligned}
			\| \pi x\|_{H(\bar\mu)} = 
			&\sup \left\{
			\phi \left( \pi x \right) \, | \,
			\phi \in Y^*, \left \Vert \phi \right \Vert_{L^2(\bar{\mu})} \le 1
			\right\} \\
			&\le
			\sup \left\{
			\phi \left( x \right) \, | \,
			\phi \in X^*_\mu \text{ linear}, \left \Vert \phi \right \Vert_{L^2(\mu)} \le 1
			\right\}  
			= \|x\|_{H(\mu)}
			\, .
		\end{aligned}
	\end{equation}
	Consequently, $\pi H(\mu) \subseteq H(\bar{\mu})$ and $\eta = \pi \upharpoonright H(\mu) : H(\mu) \to H(\bar{\mu})$ is a continuous linear map.
	Now consider the adjoint map $\eta^* : H(\bar{\mu}) \to H(\mu)$.
	For any $\phi \in Y^*$, the composition $\vert \phi \circ \pi \vert$ is a $\mu$-measurable seminorm on $X$ and hence continuous on $H(\mu)$.
	Consequently, the restriction $\phi \circ \pi$ is a continuous linear functional on $H(\mu)$.
	Thus, by the Riesz representation theorem, there is a unique $h' \in H(\mu)$ such that
	\begin{equation}
		\left< h', h \right>_{H(\mu)}
		=
		\left( \phi \circ \pi \right) \left( h \right)
		=
		\phi \left( \eta h \right)
	\end{equation}
	for all $h \in H(\mu)$.
	Since $R_\mu$ gives the identification with the dual space, \cite[Theorem 2.10.11]{src:Bogachev:GaussianMeasures} ensures that $\mu$-almost everywhere $R_\mu^{-1} h' = \phi \circ \pi = T \phi$.
	Consequently,
	\begin{equation}
		\left< R_\mu T \phi, h \right>_{H(\mu)}
		=
		\phi \left( \eta h \right)
		=
		\left< R_{\bar{\mu}} \phi, \eta h \right>_{H(\bar{\mu})}
		=
		\left<\eta^* R_{\bar{\mu}} \phi, h \right>_{H(\mu)}
	\end{equation}
	for all $h \in H(\mu)$.
	Since $Y^*$ is dense in $Y^*_{\bar{\mu}}$, we thus have that
	\begin{equation}
		\label{eq:eta*}
		\eta^* = R_\mu \circ T \circ R_{\bar{\mu}}^{-1}
	\end{equation}
	is an isometry as a composition of isometries.
	It follows that $(\eta^*)^* = \eta$ is surjective and $\eta \eta^* = \mathrm{id}_{H(\bar{\mu})}$.
	
	$(i)$:
	It is clear that $h \in \ker \eta$ if and only if $r(h) = 0$.
	Hence, given a $\mu$-measurable seminorm $p$ satisfying $p \le r$ on $X_0$, there exists a unique continuous seminorm $\bar{p}$ on $Y$ such that $p = \bar p \circ \pi$. 
	Consequently, a sequence $(h_n)_{n \in \mathbb{N}}$ in $H(\mu)$ is $p$-Cauchy if and only if $(\eta h_n)_{n \in \mathbb{N}}$ is $\bar{p}$-Cauchy.
	This provides a Banach isomorphism $n_p^{\bar{p}} : H(\mu)_p \to H(\bar{\mu})_{\bar{p}}$ with the property that
	\begin{equation}
		\kappa_{\bar{r}}^{\bar{p}} = n_p^{\bar{p}} \circ \kappa_r^p \circ \left( n_r^{\bar{r}} \right)^{-1} \, .
	\end{equation}
	Thus, $\kappa_r^p$ is injective if and only if $\kappa_{\bar r}^{\bar p}$ is.
	
	$(ii)$:
	Let $(\bar{e}_n)$ be an orthonormal basis of $H(\bar{\mu})$ and define $e_n' = \eta^* \bar{e}_n$.
	Since $\eta^*$ is an isometry, $(e_n')$ is an orthonormal basis of $\overline{\mathrm{im} \, \eta^*}$ and since $\pi$ and $\eta$ agree on $H(\mu)$, the claim $\pi e_n' = \bar e_n$ follows. Assume $(d_n)\subseteq H(\mu)$ is another orthonormal system satisfying $\pi d_n = \bar e_n$. Then $d_n-e_n'\in \operatorname{ker} \eta = ( \overline{\mathrm{im} \, \eta^*} )^\perp$.
	Hence,
	\begin{equation}
		1 = \left\Vert d_n \right\Vert_{H(\mu)}^2 = \left\Vert d_n - e_n' \right\Vert_{H(\mu)}^2 + \left\Vert e_n' \right\Vert_{H(\mu)}^2 \, .
	\end{equation}
	Since $\left\Vert e_n' \right\Vert_{H(\mu)} = 1$, it follows that $d_n = e_n'$ which proves the uniqueness.

		Now, let $(e_n)$ be any orthonormal basis extension of $(e_n')$ and
		construct $(P_n)$ and $(\bar{P}_n)$ as in \cref{thm:FiniteDimensionalApproximation} for the bases $(e_n)$ and $(\bar{e}_n)$, respectively.
		Then, defining 
		\begin{equation}
			\alpha(n) = \max \left \{ m \in \mathbb{N} \mid m \le \dim H(\bar{\mu}), e_m' \in \left \{ e_1, \dots, e_n \right \} \right \} \, ,
		\end{equation}
		and using \cref{eq:eta*}, we obtain
		\begin{equation}
			\begin{aligned}
				\pi \left( P_n x \right)
				=
				\eta \left( P_n x \right)
				&=
				\sum_{m = 1}^{\alpha(n)} R_\mu^{-1} \left( \eta^* \bar{e}_m \right) \left( x \right) \eta \eta^* \bar{e}_m \\
				&=
				\sum_{m = 1}^{\alpha(n)} R_{\bar{\mu}}^{-1} \left( \bar{e}_m \right) \left( \pi x \right) \bar{e}_m
				=
				\bar{P}_{\alpha \left( n \right)} \left( \pi x \right)
			\end{aligned}
		\end{equation}
		for $\mu$-almost every $x \in X$.
		It is clear that $\alpha$ is monotonically increasing with
		\begin{equation}
			\lim\limits_{n \to \dim H(\mu)} \alpha(n) = \dim H(\bar{\mu}) \, .
		\end{equation}
		Consequently, for any $\bar{\mu}$-measurable $F : Y \to \mathbb{R}$ and $n \in \mathbb{N}$, we have
		\begin{equation}
			\int_X F \circ \pi \, \mathrm{d} \left( \mu \circ P_n^{-1} \right)
			=
			\int_Y F \, \mathrm{d} \left( \bar{\mu} \circ \bar{P}_{\alpha \left( n \right)}^{-1} \right)
		\end{equation}
		such that \cref{eq:FunctionONBIntegrationLimit} follows whenever the respective integrals and limits exist.
		
		$(iii)$:
		Assume that $\iota_\mu^r = k \circ j$ and $(a_n)$ are chosen as stated.
		From the considerations in the proof of $(i)$ it is clear that $\iota_\mu^r = (n_r^{\bar{r}})^{-1} \circ \iota_{\bar{\mu}}^{\bar{r}} \circ \eta$.
		Hence, with $\eta \circ \eta^* = \mathrm{id}_{H(\bar{\mu})}$, we find
		\begin{equation}
			\iota_{\bar{\mu}}^{\bar{r}}
			=
			n_r^{\bar{r}} \circ \iota_\mu^r \circ \eta^*
			=
			n_r^{\bar{r}} \circ k \circ j \circ \eta^* \, .
		\end{equation}
		Thus, $\bar{k} = n_r^{\bar{r}} \circ k$ and $\bar{j} = j \circ \eta^*$ provide a factorisation in the sense of \cref{cor:GoodFactorisation} and $\vert \bar{j} \vert^2 = \eta \vert j \vert^2 \eta^*$ is a non-negative Hilbert-Schmidt operator.
		Thus the max-min principle for non-negative compact operators applies (see, e.g., \cite[Exercise 8.23]{src:DunfordSchwartz:LinearOperators2}) and since $\eta^*$ is an isometry,
		\begin{equation}
			\begin{aligned}
				\bar{\lambda}_n^2
				&=
				\max_{\substack{V\ \text{linear subspace of}\ H(\bar{\mu})\\\dim V = n}} &\min_{\substack{x\in V\\\|x\|_{H(\bar{\mu})}=1}}
				&\langle x, \eta \vert j \vert^2 \eta^* x \rangle_{H(\bar{\mu})} \\
				&=
				\max_{\substack{V\ \text{linear subspace of}\ H(\bar{\mu})\\\dim V = n}} &\min_{\substack{y\in \eta^* V\\\|y\|_{H(\mu)} =1}}
				&\langle y, \vert j \vert^2 y \rangle_{H(\mu)} \\
				&\le
				\max_{\substack{W\ \text{linear subspace of}\ H(\mu)\\\dim W = n}} &\min_{\substack{y \in W\\\|y\|_{H(\mu)}=1}}
				&\langle y, \vert j \vert^2 y \rangle_{H(\mu)}
				=
				\lambda_n^2 \, .
			\end{aligned}
		\end{equation}
		Hence, $\sum_{n} \bar{\lambda}_n^2 a_n \le \sum_{n} \lambda_n^2 a_n < \infty$.
	\end{proof}
	We now obtain the integrability and convergence in full generality.
	\begin{corollary}\label{finalcor}
		Let $\mu$ be a centred Radon Gaussian measure on a locally convex space $X$ and let $p, q$ be two $\mu$-measurable seminorms on $X$.
		Furthermore, suppose that
		\begin{enumerate}[(i)]
			\item the natural map $\iota_\mu^{p+q}: H(\mu) \to H(\mu)_{p+q}$ is absolutely $2$-summing,
			\item the natural map $H(\mu)_{p+q} \to H(\mu)_p$ is injective and
			\item for some linear subspace $X_0 \subseteq X$ of full $\mu$-measure on which $p$ and $q$ are seminorms, the $(p+q)$-closure $\mathrm{cl}^{X_0}_{p+q} H(\mu)$ has full $\mu$-measure.
		\end{enumerate}
		Then, for any continuous and submultiplicative function $f : [0,\infty) \to [0,\infty)$ satisfying \cref{eq:superquad}, we have
		\begin{equation}
			\int_X \exp \left[ - f \left( p \left( x \right) \right) + q \left( x \right)^2 \right] \mathrm{d} \mu \left( x \right) < \infty \, . 
		\end{equation}
		Assume in addition that the natural map $\iota_\mu^{p+q} : H(\mu) \to H(\mu)_{p+q}$ admits a factorisation $\iota_\mu^{p+q} = k \circ j$ as in \cref{cor:GoodFactorisation} with the property that
		\begin{equation}\label{eq:evsumming}
			\sum_n \lambda_n^2 \ln \left( n + 1 \right) < \infty \, .
		\end{equation}
		where $(\lambda_n)$ denotes the (possibly finite) sequence of eigenvalues of $|j|$ sorted in descending order.
		Then, there exists an orthonormal basis $(e_n)$ of $H(\mu)$ such that
		\begin{equation}
			\begin{aligned}
				\lim_{n \to \mathrm{dim} \, H(\mu)} \int_X \exp &\left[ -f\left(p \left( x \right)\right) + \alpha q \left( x \right)^2 \right] \mathrm{d} \left( \mu \circ P_n^{-1} \right) \left( x \right) \\
				&=
				\int_X \exp \left[ -f\left(p \left( x \right)\right) + \alpha q \left( x \right)^2 \right] \mathrm{d} \mu \left( x \right)
				< \infty
			\end{aligned}
		\end{equation}
		for all $\alpha > 0$.
	\end{corollary}
	\begin{proof}
		The statement follows immediately by combining \cref{thm:FiniteDimensionalApproximation,thm:FiniteDimensionalConvergence,thm:sepBSextension}.
	\end{proof}
	\begin{remark}
		$\iota_\mu^{p+q}$ is trivially absolutely $2$-summing in the case that $H(\mu)$ is finite-dimensional.
	\end{remark}
	\begin{remark}
		If $p$ and $q$ are continuous, condition $(iii)$ is always true by \cref{lem:CMhilbert}.
		If in addition $X$ is a nuclear space, condition $(i)$ and \cref{eq:evsumming} are automatically satisfied.
	\end{remark}
	\begin{remark}
		In case $p+q$ is a sequentially continuous seminorm, the assumption $\mu(\mathrm{cl}^{X_0}_{p+q}H(\mu)) = 1$ is true for any linear subspace $X_0 \subseteq X$ of full measure on which $p$ and $q$ are seminorms and follows from \cref{thm:FiniteDimensionalApproximation}.
		Hence, this restriction can be viewed as a natural generalisation of sequential continuity.
	\end{remark}
	We conclude this \lcnamecref{sec:MeasurableSeminorms} with the promised \lcnamecref{ex2}.
	\begin{example}\label{ex2}
		Let $X=\mathbb{R}^{\mathbb{N}}$ and consider $\mu_\sigma$ as in \cref{exsimple} with $\sigma\in\ell^2$.
		Using that $X^*$ can be identified with the set of all finite sequences in $\mathbb{R}^{\mathbb N}$, one readily checks that
		$H(\mu_\sigma) = \big\{ x\in X : \sum_{n=1}^{\infty}\sigma_n^{-2} x_n^2 < \infty\big\}$.
		By the density of compactly supported sequences in $\ell^2$ and monotonicity of the $\ell^r$-norms, we find
		$\operatorname{cl}^{\ell^2}_{p_2+p_r}H(\mu) = \ell^2$ for any $r\ge 2$.
		The injectivity assumption follows from \cref{rem:Lr}.
		It remains to prove the absolute $2$-summability of the natural map $\iota_{\mu_\sigma}^{p_2+p_r} : H(\mu) \to H(\mu)_{p_2 + p_r}$.
		Let $(e_n)_{n\in\mathbb{N}}$ denote the usual orthonormal basis in $\ell^2$. Then, $(\sigma_n e_n)_{n\in\mathbb N}$ is an orthonormal basis of $H(\mu_\sigma)$.
		Hence, it follows that
		\[ 
		\sum_{k = 1}^\infty \left(p_2(\sigma_ne_n) + p_r(\sigma_ne_n)\right)^2
		\le
		4 \sum_{k = 1}^\infty p_2(\sigma_n e_n)^2
		\le
		4\sum_{k = 1}^\infty \sigma_n^2<\infty.
		\]
		Therefore, $\iota_{\mu_\sigma}^{p_2+p_r}$ is Hilbert--Schmidt and hence absolutely 2-summing. By \cref{finalcor}, this implies
		\begin{equation}
			\int_{\mathbb{R}^{\mathbb{N}}}
			\exp \left[ - p_r \left( x \right)^{2+\epsilon} + \alpha p_2 \left( x \right)^2 \right]
			\mathrm{d} \mu_\sigma \left( x \right)
			< \infty \, ,
		\end{equation}
		for all $\alpha, \epsilon > 0$.
	\end{example}
\subsection*{Acknowledgements}
B.H. and J.Z. thank Daan van Dijk and David Hasler for valuable discussions on the topic.
D.W.J. and J.Z. were supported by the Deutsche Forschungsgemeinschaft (DFG) under Grant No. 406116891 within the Research Training Group RTG 2522/1.
J.Z. was further supported by the DFG under Grant No. 398579334 (Gi328/9-1).
\appendix
\bibliographystyle{halpha-abbrv}
\bibliography{bibliography}
\end{document}